\documentclass[a4paper,12pt]{article}
\usepackage{amsfonts,amsmath,url}
\usepackage{color} 
\usepackage{soul}  

\newcommand{\nmax}{145}

\newcommand{\dash}[1]{--~#1~--}

\newcommand{\RN}{{\,RN\,}}
\newcommand{\RNe}{\,RN_e\,}
\newcommand{\RNd}{\,RN_d\,}
\newcommand{\RU}{{\,RU\,}}
\newcommand{\RD}{{\,RD\,}}
\newcommand{\RZ}{{\,RZ\,}}
\newcommand{\FastMult}{{\,Fast2Mult\,}}
\newcommand{\FastSum}{{\,Fast2Sum\,}}

\newcommand{\DblMult}{\,\mbox{\it DblMult}\,}
\newtheorem{theorem}{Theorem}

\newtheorem{algorithm}{Algorithm}

\newcommand{\ulp}{\mbox{\rm ulp}}
\newcommand{\hulp}{{\frac12}\mbox{\rm ulp}}

\newcommand{\mb}[1]{\makebox(0,0){\rule[-1ex]{0ex}{3ex}#1}}

\newenvironment{proof}%
{\begin{trivlist}\item[]{\bf Proof:}\nopagebreak[4]}%
{\mbox{}\hfill$\Box$\end{trivlist}}

\if@twoside \oddsidemargin -10pt \evensidemargin 2pt \marginparwidth
 67pt
\else \oddsidemargin -4pt \evensidemargin -4pt
\fi
\title{Computing Integer Powers in\\Floating-Point Arithmetic}

\author{\small Peter Kornerup \and \small Vincent Lef\`evre \and  \small Jean-Michel Muller
\thanks{Peter~Kornerup is with SDU, Odense, Denmark;
Vincent Lef\`evre and Jean-Michel~Muller are with
  Laboratoire LIP, CNRS/ENS Lyon/INRIA/Univ. Lyon 1, Lyon, France.
  }}

\date{{\small May 2007\\
This is LIP Research Report number RR2007-23\\
Ceci est le Rapport de Recherches num\'ero RR2007-23 du LIP\\
Laboratoire LIP, CNRS/ENS Lyon/INRIA/Univ. Lyon 1, Lyon, France.}}
\begin{document}
\maketitle
\sloppy

\begin{abstract}
We introduce two algorithms for accurately evaluating powers to a
positive integer in floating-point arithmetic, assuming a
\emph{fused multiply-add} (fma) instruction is available. We show
that our log-time algorithm always produce faithfully-rounded
results, discuss the possibility of getting correctly rounded
results, and show that results correctly rounded in double precision
can be obtained if extended-precision is available with the
possibility to round into double precision (with a single rounding).
\end{abstract}

\section{Introduction}

We deal with the implementation of the integer power function in
floating-point arithmetic.  In the following, we assume a radix-$2$
floating-point arithmetic that follows the IEEE-754 standard for
floating-point arithmetic. We also assume that a fused multiply-and-add
(fma) operation is available, and that the input as well as the output
values of the power function are not subnormal numbers, and are below the
overflow threshold (so that we can focus on the powering of the
significands only).

An important case dealt with in the paper will be the case when an
internal format, wider than the target format, is available. For
instance, to guarantee \dash{in some cases} correctly rounded integer
powers in double precision arithmetic, we will have to assume that a
double-extended precision is available. The examples will consider
that it has a 64-bit precision, which is the minimum required by the
IEEE-754 standard.

The IEEE-754 standard~\cite{IEEE85} for radix-2 floating-point
arithmetic (and its follower, the IEEE-854 radix-independent
standard~\cite{IEEE87}) require that the four arithmetic operations
and the square root should be correctly rounded. In a floating-point
system that follows the standard, the user can choose an
\emph{active rounding mode} from:
\begin{itemize}
                \item rounding towards $-\infty$: $\RD (x)$ is the largest
                machine number less than or equal to $x$;

                \item rounding towards $+\infty$: $\RU (x)$ is the smallest
                machine number greater than or equal to $x$;

                \item rounding towards $0$: $\RZ (x)$ is equal to
                $\RD (x)$ if $x \geq 0$, and to $\RU (x)$ if $x < 0$;

              \item rounding to nearest: $\RN (x)$ is the machine number
                that is the closest to $x$ (with a special convention if
                $x$ is exactly between two machine numbers: the chosen
                number is the ``even'' one, i.e., the one whose last
                significand bit is a zero).
\end{itemize}

When $a \circ{} b$ is computed, where $a$ and $b$ are floating-point
numbers and $\circ{}$ is $+$, $-$, $\times$ or $\div$, the returned
result is what we would get if we computed $a \circ{} b$ exactly,
with ``infinite'' precision and rounded it according to the active
rounding mode. The default rounding mode is round-to-nearest. This
requirement is called \emph{correct rounding}. Among its many
interesting properties, one can cite the following result (the first
ideas that underlie it go back to M\o{}ller~\cite{Moller1965}).

\begin{theorem}[Fast2Sum algorithm](Theorem C of~\cite{Knu98},
page 236). Assume the radix~$r$ of the floating-point system being
considered is less than or equal to~$3$, and that the used
arithmetic provides correct rounding with rounding to nearest.
Let $a$ and $b$ be floating-point numbers, and assume that the
exponent of $a$ is larger than or equal to that of $b$. The
following algorithm computes two floating-point numbers $s$ and $t$
that satisfy:
\begin{itemize}
  \item $s+t = a+b$ exactly;
  \item $s$ is the floating-point number that is closest to $a+b$.
\end{itemize}
\end{theorem}
\begin{algorithm}[Fast2Sum(a,b)]~\\[-2ex]
 \[
 \begin{array}{lll}
  s &:=& \RN(a + b);\\
  z &:=& \RN(s - a);\\
  t &:=& \RN(b - z);
 \end{array}
 \]
\end{algorithm}

If no information on the relative orders of magnitude of $a$ and $b$
is available, there is an alternative algorithm introduced by
Knuth~\cite{Knu98}. It requires $6$ operations instead of $3$ for
the Fast2Sum algorithm, but on any modern computer, the $3$
additional operations cost significantly less than a comparison
followed by a branching.


Some processors (e.g., the IBM PowerPC or the Intel/HP
Itanium~\cite{CorHarTan2002}) have a \emph{fused multiply-add} (fma)
instruction that allows to compute $ax\pm b$, where $a$, $x$ and $b$
are floating-point numbers, with one final rounding only. This
instruction allows one to design convenient software algorithms for
correctly rounded division and square root. It also has the
following interesting property. From two input floating-point
numbers $a$ and $b$, the following algorithm computes $c$ and $d$
such that $c + d = ab$, and $c$ is the floating-point number that is
nearest $ab$. \newpage
\begin{algorithm}[Fast2Mult(a,b)]~\\[-2ex]
 \[
 \begin{array}{lll}
     c &:=& \RN(ab); \\
     d &:=& \RN(ab - c);
 \end{array}
 \]
\end{algorithm}
Performing a similar calculation without a fused multiply-add
operation is possible~\cite{Dekker71} but requires $17$
floating-point operations instead of $2$.

Algorithms Fast2Sum and Fast2Mult both provide double-precision results of
value $(x+y)$ represented in the form of pairs $(x,y)$. In the following we
need product of numbers represented in this form. However, we will be
satisfied with approximations to the product, discarding terms of the order
of the product of the two low-order terms. Given two double-precision
operands $(a_h+a_l)$ and $(b_h+b_l)$ the following algorithm $\DblMult$
computes $(x,y)$ such that $(x+y)=[(a_h+a_l)(b_h+b_l)](1+\delta)$ where the
relative error $\delta$ is discussed in Section~3 below.

\begin{algorithm}[DblMult($a_h$,$a_l$,$b_h$,$b_l$)]~\\[-2ex]
 \[
 \begin{array}{lll}
     t &:=& \RN(a_l b_h);\\
     s &:=& \RN(a_h b_l + t);\\
     (x',u) &:=& \FastMult(a_h,b_h); \\
     (x'',v) &:=& \FastSum(x',s); \\
     y' &:=& \RN(u+v);\\
     (x,y) &:=& \FastSum(x'',y');
 \end{array}
 \]
\end{algorithm}
Note that the condition for applying $\FastSum$ is satisfied.

\section{The two algorithms}

We now give two algorithms for accurately computing $x^n$, where $x$
is a floating-point number, and $n$ is an integer greater than or
equal to $1$. We assume that an fma instruction is available, as it is used
in $\FastMult$ and thus implicitly also in $\DblMult$.

The first (${\cal O}(n)$ time) algorithm is derived from the
straightforward, $(n-1)$-multiplication, algorithm. It is simple to analyze
and will be faster than the other one if $n$ is small.

\begin{algorithm}[LinPower($x,n$), $n \geq 1$]~\\[-2ex]
\[
\begin{array}{lll}
(h,l) := (x,0);\\
\mbox{\bf for}\;i\;\mbox{\bf from}\;2\;\mbox{\bf to}\;n\;\mbox{\bf do}\\
\hspace{2ex}(h,v) := \FastMult($h,x$);\\
\hspace{6.5ex}l     := \RN(l\,x+v);\\
\mbox{\bf end do};\\
\mbox{\bf return} \;(h,l);
\end{array}
\]
\end{algorithm}
where the low order terms are accumulated with appropriate weights using a
Horner scheme evaluation. Algorithm LinPower uses $3n-3$ floating-point
operations.

The second (${\cal O}(\log(n))$-time) algorithm is based on successive
squarings.

\newcommand{\ind}{\hspace{2ex}}
\begin{algorithm}[LogPower($x,n$), $n \geq 1$]~\\[-2ex]
\[
\begin{array}{lll}
i:=n;\\
(h,l):=(1,0);\\
(u,v):=(x,0);\\
\mbox{\bf while}\;i>1\;\mbox{\bf do}\\
\ind\mbox{\bf if}\;(i\;\mbox{\bf mod}\;2)=1\;\mbox{\bf then}\\
\ind\ind(h,l) := \DblMult(h,l,u,v);\\
\ind\mbox{\bf end};\\
\ind(u,v)   := \DblMult(u,v,u,v);\\
\ind i      := \lfloor i/2 \rfloor;\\
\mbox{\bf end do};\\
\mbox{\bf return}\;\DblMult(h,l,u,v);
\end{array}
\]
\end{algorithm}

Due to the approximations performed in algorithm $\DblMult$, terms
corresponding to the product of low order terms are not included. A
thorough error analysis is performed below. The number of floating-point
operations used by the LogPower algorithm is between $11(1+\left\lfloor
  \log_2(n)\right\rfloor)$ and $11(1+2\left\lfloor\log_2(n)
\right\rfloor)$, whereas for LinPower it is $3(n-1)$. Hence, LogPower will
become faster than LinPower for values of $n$ around $30$ (but counting the
floating-point operations only gives a rough estimate, the actual threshold
will depend on the architecture and compiler).

\section{Error analysis}
We will use the following result.

\begin{theorem}[Theorem 2.2 of~\cite{Hig02}, p. 38]
Assume a radix-$r$ floating-point system $F$, with precision $p$. If
$x \in \mathbb{R}$ lies in the range of $F$, then
\[
\RN(x) = x(1+\delta),\;|\delta| < \frac{1}{2} r^{-p+1}.
\]
\end{theorem}

\subsection{Error of function DblMult}

\begin{theorem}
\label{thm-dblmult}
   Let $\epsilon = 2^{-p}$, where $p$ is the precision of the
   radix-$2$ floating-point system used.
   If $|a_l| \leq 2^{-p}|a_h|$ and $|b_l| \leq 2^{-p}|b_h|$ then the
   returned value $(x,y)$ of function DblMult$(a_h,a_l,b_h,b_l)$
   satisfies
   \[
   x+y = (a_h+a_l)(b_h+b_l)(1+\eta),
   \]
   with
   \[
   |\eta| \leq 6 \epsilon^2 + 16 \epsilon^3 + 17 \epsilon^4 + 11
   \epsilon^5 + 5 \epsilon^6 + \epsilon^7.
   \]
\end{theorem}

\noindent\textbf{Notes:}
\begin{enumerate}
\setlength{\itemsep}{-5pt}
  \item as soon as $p \geq 5$, we have $|\eta| \leq 7\epsilon^2$;
  \item in the case of single precision $(p=24)$, $|\eta| \leq
  6.000001 \epsilon^2$;
  \item in the case of double precision $(p=53)$, $|\eta| \leq \left(6 +
  2\times{}10^{-15}\right)\epsilon^2.$
\end{enumerate}

\begin{proof} Following the notation in Algorithm~5, with $\epsilon_i$'s
  being variables of absolute value less than $\epsilon$, we have
\begin{eqnarray*}
x+y &=& x'' + \RN(u+v)\\
&=& x'' + (u+v)(1+\epsilon_1)\\
&=& (x'' + v) + u + u\epsilon_1 + v\epsilon_1 \\
&=& x' + s  + u + u\epsilon_1 + v\epsilon_1 \\
&=& a_hb_h + s + u\epsilon_1 + v\epsilon_1 \\
&=& a_hb_h + [a_hb_l + (a_lb_h)(1+\epsilon_3)](1+\epsilon_2) + u\epsilon_1 + v\epsilon_1 \\
&=& a_hb_h + a_hb_l + a_lb_h + a_hb_l\epsilon_2 + a_lb_h\epsilon_2 +
a_lb_h\epsilon_2\epsilon_3+a_lb_h\epsilon_3  + u\epsilon_1 +
v\epsilon_1.
\end{eqnarray*}

We also have $a_l=\epsilon_4 a_h$, $b_l=\epsilon_5 b_h$,
$u=\epsilon_6 a_hb_h$, and
\begin{eqnarray*}
v &=& \epsilon_7 (x' + s) \\
&=& \epsilon_7\left(a_hb_h(1+\epsilon_8)+[a_hb_l
  + a_lb_h(1+\epsilon_3)](1+\epsilon_2)\right)\\
&=& \epsilon_7\left(a_hb_h(1+\epsilon_8)+[\epsilon_5a_hb_h
   + \epsilon_4a_hb_h(1+\epsilon_3)](1+\epsilon_2)\right)\\
&=& \epsilon_7a_hb_h\left(1 + \epsilon_8 + \epsilon_5+\epsilon_2\epsilon_5
    + \epsilon_4+\epsilon_2\epsilon_4 + \epsilon_3\epsilon_4
    + \epsilon_2\epsilon_3\epsilon_4\right)\\
&=& \eta_1 a_hb_h,
\end{eqnarray*}
with $|\eta_1| \leq \epsilon + 3\epsilon^2 + 3 \epsilon^3 + \epsilon^4.$
Hence
\begin{eqnarray*}
x+y &=& a_hb_h + a_hb_l + a_lb_h + (a_lb_l - \epsilon_4\epsilon_5 a_hb_h)
  + a_hb_h (\epsilon_2\epsilon_5 + \epsilon_2\epsilon_4
  + \epsilon_2\epsilon_3\epsilon_4 + \epsilon_3\epsilon_4
  + \epsilon_1\epsilon_6 + \eta_1 \epsilon_1)\\
&=& (a_h+a_l)(b_h+b_l) + a_hb_h \eta_2,
\end{eqnarray*}
with $|\eta_2| \leq 6 \epsilon^2 + 4\epsilon^3 + 3\epsilon^4 + \epsilon^5$.

Now, from $a_h=(a_h+a_l)(1+\epsilon_9)$ and
$b_h=(b_h+b_l)(1+\epsilon_{10})$ we deduce
\[
x+y = (a_h+a_l)(b_h+b_l)(1+\eta),
\]
with $\eta = (1+\epsilon)^2\eta_2$, which gives $|\eta| \leq 6 \epsilon^2 +
16 \epsilon^3 + 17 \epsilon^4 + 11 \epsilon^5 + 5 \epsilon^6 + \epsilon^7.$
\end{proof}

\subsection{Error of algorithm LogPower}

\begin{theorem}
\label{thm-LogPower}
The two values $h$ and $l$ returned by algorithm LogPower satisfy
\[
h+l = x^n(1+\alpha),
\]
with
\[
(1-|\eta|)^{n-1} \leq 1+\alpha \leq (1+|\eta|)^{n-1}
\]
where $|\eta| \leq 6 \epsilon^2 + 16 \epsilon^3 + 17 \epsilon^4 + 11
\epsilon^5 + 5 \epsilon^6 + \epsilon^7$ is the same value as in
Theorem~\ref{thm-dblmult}.
\end{theorem}

\begin{proof}
  Algorithm LogPower computes approximations to powers of $x$, using
  $x^{i+j} = x^ix^j$. By induction, one easily shows that the approximation
  to $x^k$ is of the form $x^k(1+\beta_k)$, where $(1-|\eta|)^{k-1} \leq
  (1+\beta_k) \leq (1+|\eta|)^{k-1}$. If we call $\eta_{i+j}$ the relative
  error (obtained from Theorem~\ref{thm-dblmult}) when multiplying together
  the approximations to $x^i$ and $x^j$, the induction follows from
\[
(1-\eta)^{i-1}(1-\eta)^{j-1}(1-\eta) \leq
\left(x^i(1+\beta_i)\right)\left(x^j(1+\beta_j)\right)(1+\eta_{i+j})
\leq (1+\eta)^{i-1}(1+\eta)^{j-1}(1+\eta).
\]
\end{proof}
%

Table~\ref{table-logpower-double} gives bounds on $|\alpha|$ for
several values of $n$ (note that the bound is an increasing value of
$n$), assuming the algorithm is used in double precision.

Define the \emph{significand} of a non-zero real number $u$ to be
\[
\frac{u}{2^{\left\lfloor \log_2 |u| \right\rfloor}}.
\]
Define $\alpha_{max}$ as the bound on $|\alpha|$ obtained for a given value
of $n$.  From
\[
x^n(1 - \alpha_{max}) \leq h+l \leq x^n(1 + \alpha_{max}),
\]
we deduce that the significand of $h+l$ is within $2 \alpha_{max}$
from $x^n/2^{\left\lfloor \log_2 |h+l| \right\rfloor}$.
  From the results given in Table~\ref{table-logpower-double}, we
deduce that for all practical values of $n$ the significand of $h+l$
is within much less than $2^{-53}$ from $x^n/2^{\left\lfloor \log_2
|h+l| \right\rfloor}$ (indeed, to get $2 \alpha_{\max}$ larger that
$2^{-53}$, we need $n > 2^{49}$). This means that $\RN(h+l)$ is
within less than one ulp from $x^n$, hence

\begin{theorem}
  If algorithm LogPower is implemented in double precision, then $\RN(h+l)$
  is a faithful rounding of $x^n$, as long as $n \leq 2^{49}$.
\end{theorem}

\begin{table}[h]
\[
\begin{array}{|r|r||r|r|}
\hline
n & -\log_2(\alpha_{max}) & n & -\log_2(\alpha_{max}) \\
\hline\hline
3 & 102.41 & 1000 & 93.45 \\
\hline
4 & 101.83 & 10{,}000 & 90.12 \\
\hline
5 & 101.41 & 100{,}000 & 86.80 \\
\hline
10 & 100.24 & 1{,}000{,}000 & 83.48 \\
\hline
20 & 99.16 & 10{,}000{,}000 & 80.16 \\
\hline
30 & 98.55 & 100{,}000{,}000 & 76.83 \\
\hline
40 & 98.12 & 2^{32}       & 71.41 \\
\hline
50 & 97.80 & & \\
\hline
100 & 96.78 & & \\
\hline
200 & 95.77 & & \\
\hline
\end{array}
\]
\caption{Binary logarithm of the relative accuracy
($-\log_2(\alpha_{max})$), for various values of $n$ assuming
algorithm LogPower is used in double precision.}
\label{table-logpower-double}
\end{table}

Moreover, for $n \leq 10^8$, $\RN(h+l)$ is within $0.50000007$ ulps
from the exact value: we are very close to correct rounding (indeed,
we almost always return a correctly rounded result), yet we cannot
guarantee correct rounding, even for the smallest values of $n$.
This requires a much better accuracy, as shown in
Section~\ref{section-correct-rounding}.  To guarantee a correctly
rounded result in double precision, we will need to run algorithm
LogPower in double-extended precision.
Table~\ref{table-logpower-extended} gives bounds on $|\alpha|$ for
several values of $n$ assuming the algorithm is realized in
double-extended precision. As expected, we are 22 bits more
accurate.

\begin{table}[h]
\[
\begin{array}{|r|r||r|r|}
\hline
n & -\log_2(\alpha_{max}) & n & -\log_2(\alpha_{max}) \\
\hline\hline
3 & 124.41 & 1000 & 115.45 \\
\hline
4 & 123.83 & 10{,}000 & 112.12 \\
\hline
5 & 123.41 & 100{,}000 & 108.80 \\
\hline
10 & 122.24 & 1{,}000{,}000 & 105.48 \\
\hline
20 & 121.16 & 10{,}000{,}000 & 102.16 \\
\hline
30 & 120.55 & 100{,}000{,}000 & 98.83 \\
\hline
40 & 120.12 & 2^{32}       & 93.41 \\
\hline
50 & 119.80 & & \\
\hline
100 & 118.78 & & \\
\hline
200 & 117.77 & & \\
\hline
\end{array}
\]
\caption{Binary logarithm of the relative accuracy
  ($-\log_2(\alpha_{max})$), for various values of $n$ assuming
  algorithm LogPower is implemented in double-extended precision.}
\label{table-logpower-extended}
\end{table}

\subsection{Error of algorithm LinPower}

Define $h_i$, $v_i$, $l_i$ as the values of variables $h$, $v$ and $l$ at
the end of the loop of index $i$ of the algorithm. Define $\hat{l}_i$ as
the value variable $l_i$ would have if the instructions $l := \RN(lx+v)$
were errorless (that is, if instead we had $l := (lx+v)$ exactly):
\begin{equation}
\label{def-hatli}
\hat{l}_i = v_i + v_{i-1} x + v_{i-2} x^2 + v_{i-3} x^3 + \cdots + v_2 x^{i-2}.
\end{equation}

Initially let $h_1=x$, $v_1 = l_1 = 0$.
By induction, one can easily show that
\begin{equation}
\label{sum-vi-xj}
x^i = h_i + v_i + v_{i-1}x + v_{i-2}x^2 + v_{i-3}x^3 +\cdots{} + v_2x^{i-2},
\end{equation}
hence we have
\[
x^i = h_i + \hat{l}_i.
\]
The algorithm only computes an approximation ${l}_i$ to $\hat{l}_i$. To
evaluate the error of the algorithm, we must therefore estimate the
distance between ${l}_i$ and $\hat{l}_i$. We have ${l}_1 = \hat{l}_1 = 0$,
and ${l}_2 = \hat{l}_2 = v_2$ exactly. Define $\epsilon_i$ as the number of
absolute value less than $\epsilon = 2^{-p}$ such that
\[
{l}_i = \RN({l}_{i-1}x+v_i) =
({l}_{i-1}x+v_i)(1+\epsilon_i).
\]
We have $l_3 = \hat{l}_3(1+\epsilon_3)$, and by induction, we find
for $i \geq 4$, using $v_i = \hat{l}_i - \hat{l}_{i-1}x$:
\begin{eqnarray}\nonumber
l_i &=& \hat{l}_i(1+\epsilon_i)\\\nonumber
&+& \hat{l}_{i-1}\epsilon_{i-1}x(1+\epsilon_i) \\
\label{li-vs-hat-li}
&+&  \hat{l}_{i-2}\epsilon_{i-2}x^2(1+\epsilon_{i-1})(1+\epsilon_i)\\\nonumber
&\vdots& \\\nonumber
&+& \hat{l}_3\epsilon_{3}x^{i-3}(1+\epsilon_{4})(1+\epsilon_{5})
    \cdots{}(1+\epsilon_{i-1})(1+\epsilon_i).
\end{eqnarray}
To derive a useful bound from this result, we must make a simplifying
hypothesis. We know that $|v_i| \leq \epsilon |h_i|.$ We assume $h_i$ is
close enough to $x^i$, so that
\[
|v_i| \leq 2 \epsilon |x|^i
\]
(this means that our estimate for $x^n$ will become wrong when the
algorithm becomes very inaccurate for $x^i$, $i \leq n$).
From~(\ref{def-hatli}), we therefore have:
\[
|\hat{l}_i| \leq  2 (i-1) \epsilon |x|^i,
\]
from which, using (\ref{li-vs-hat-li}), we deduce
\[
l_n = \hat{l}_n + \eta,
\]
where
\begin{equation}
  |\eta| \leq 2 |x|^n \epsilon^2 \left[ (n-1) + (n-2)(1+\epsilon) +
    (n-3)(1+\epsilon)^2 + \cdots{} + 2(1+\epsilon)^{n-3}\right].
\end{equation}
This gives the following result\\

\begin{theorem}[Accuracy of algorithm LinPower]
  If for $i < n, |v_i| \leq 2^{1-p}|x|^i$, the final computed values $h_n$
  and $l_n$ of the variables $h$ and $l$ of the algorithm satisfy
\[
h_n + l_n = x^n (1+\alpha),
\]
where
$
|\alpha| \leq 2 \epsilon^2 \left[(n-1) + (n-2)(1+\epsilon) +
(n-3)(1+\epsilon)^2 + \cdots{} + 2(1+\epsilon)^{n-3}\right].
$
\end{theorem}

Let us try to compute an estimate of the coefficient $\gamma = (n-1) +
(n-2)(1+\epsilon) + (n-3)(1+\epsilon)^2 + \cdots{} + 2(1+\epsilon)^{n-3}$
in $\alpha$.

Define a function
\[
\varphi(t) = t^{n-1}+(1+\epsilon)t^{n-2} + (1+\epsilon)^2t^{n-3}+
\cdots{} + (1+\epsilon)^{n-3}t^2.
\]
One can notice that $\gamma = \varphi'(1)$, so that if we are able
to find a simple formula for $\varphi(t)$ we will be able to deduce
a formula for $\gamma$. We have
\[
\varphi(t) = (1+\epsilon)^{n-1}\left[
\left(\frac{t}{1+\epsilon}\right)^{n-1} +
\left(\frac{t}{1+\epsilon}\right)^{n-2} + \cdots{}
+\left(\frac{t}{1+\epsilon}\right)^{2}\right],
\]
hence
\[
\varphi(t) = (1+\epsilon)^{n-1}\left[
\frac{\left(\frac{t}{1+\epsilon}\right)^n-1}{\frac{t}{1+\epsilon}-1}-\frac{t}{1+\epsilon}-1
\right].
\]
Thus
\[
\varphi'(t) = (1+\epsilon)^{n-2}\left[
\frac{(n-1)\left(\frac{t}{1+\epsilon}\right)^n-n\left(\frac{t}{1+\epsilon}\right)^{n-1}+1}{\left(\frac{t}{1+\epsilon}-1\right)^2}-1
\right],
\]

Hence a bound on the value of $|\alpha|$ is,
\[
|\alpha| \leq 2 \epsilon^2 (1+\epsilon)^{n-2}\left[
\frac{(n-1)\left(\frac{1}{1+\epsilon}\right)^n-n\left(\frac{1}{1+\epsilon}\right)^{n-1}+1}{\left(\frac{1}{1+\epsilon}-1\right)^2}-1
\right] \approx (n^2-n-2)\epsilon^2.
\]

Table~\ref{table-alpha-linpower} gives the obtained bound on $|\alpha|$ for
several values of $n$, assuming double precision ($\epsilon = 2^{-53}$).
That table shows that as soon as $n$ is larger than a few units, algorithm
LinPower is less accurate than algorithm LogPower.


\begin{table}[h]
\[
\begin{array}{|r|l|}
\hline n & -\log_2(\alpha_{max}) \\
\hline\hline
3 & 104.00 \\
\hline
4 & 102.68 \\
\hline
5 & 101.83 \\
\hline
10 & 99.54 \\
\hline
20 & 97.43 \\
\hline
30 & 96.23 \\
\hline
100 & 92.72 \\
\hline
\end{array}
\]
\caption{Binary logarithm of the relative accuracy
  ($-\log_2(\alpha_{max})$), for various values of $n$ assuming
  algorithm LinPower is implemented in double precision.}
\label{table-alpha-linpower}
\end{table}

\section{Correct rounding}
\label{section-correct-rounding}

In this section we consider algorithm LogPower only: first because
it is the fastest for all reasonable values of $n$, second because
it is the only one for which we have certain error bounds (the error
bounds of algorithm LinPower are approximate only). And if needed,
specific algorithms could be designed for small values of $n$. We
are interested in getting correctly rounded results in double
precision. To do so, we assume that we perform algorithm LogPower in
double extended precision.  The algorithm returns two
double-extended numbers $h$ and $l$ such that
\[
x^n (1 - \alpha_{max}) \leq h+l \leq x^n (1 + \alpha_{max}),
\]
where $\alpha_{max}$ is given in Table~\ref{table-logpower-extended}.

In the following we will need to distinguish two roundings, e.g.,
$\RNe$ means round-to-nearest in extended double precision and
$\RNd$ is round-to-nearest in double precision. Let $\ulp(\cdot)$
denote ``unit-in-last-position'' such that $|x-\RN(x)| \leq \hulp(x)$.

V.~Lef\`evre introduced a new method for finding hardest-to-round cases for
evaluating a regular function~\cite{Lef2000,Lef99}. That method allowed
Lef\`evre and Muller to give such cases for the most familiar elementary
functions~\cite{LefevreMuller2001a}. Recently, Lef\`evre adapted his method
to the case of functions $x^n$ and $x^{1/n}$, when $n$ is an integer. For
instance, in double-precision arithmetic, the hardest to round case for
function $x^{51}$ corresponds to
\[
x = 1.0100010111101011011011101010011111100101000111011101
\]\\[-5ex]
we have
\[
\begin{array}{lll}
x^{51}  &=&
\underbrace{1.1011001110100100011100100001100100000101101011101110}_{53\mbox{~bits}}\mbox{~}1
\\
        & & \underbrace{0000000000 \cdots{}0000000000}_{59\mbox{~zeros}}100\cdots
        \times 2^{17}
\end{array}
\]
which means that $x^n$ is extremely close to the exact middle of two
consecutive double-precision numbers. There is a run of $59$ consecutive
zeros after the rounding bit. This case is the worst case for all values of
$n$ between $3$ and $\nmax$.  Table~\ref{worst-cases-double} gives the
maximal length of the chains of identical bits after the rounding bit
for $3 \leq n \leq \nmax$.\\[-4ex]
\begin{table}[h]
\[\footnotesize
\begin{array}{|l|l|}
\hline
n & \ \begin{array}{l}
          \mbox{Number of identical bits} \\
          \mbox{after the rounding bit}
      \end{array} \\
\hline\hline  32  &  48
\\ \hline 76, 81, 85  &  49
\\ \hline 9, 15, 16, 31, 37, 47, 54, 55, 63, 65, 74, 80, 83, 86,
 105, 109, 126, 130  &  50
\\ \hline
\begin{array}{l}
  10, 14, 17, 19, 20, 23, 25, 33, 34, 36, 39, 40, 43, 46, 52, 53, \\
  72, 73, 75, 78, 79, 82, 88, 90, 95, 99, 104, 110, 113, 115, 117, \\
  118, 119, 123, 125, 129, 132, 133, 136, 140
\end{array} &  51
\\ \hline
\begin{array}{l}
  3, 5, 7, 8, 22, 26, 27, 29, 38, 42, 45, 48, 57, 60, 62, 64, 68, 69, \\
  71, 77, 92, 93, 94, 96, 98, 108, 111, 116, 120, 121, 124, 127, 128, \\
  131, 134, 139, 141
\end{array} &  52
\\ \hline 6, 12, 13, 21, 58, 59, 61, 66, 70, 102, 107, 112, 114,
 137, 138, 145  &  53
\\ \hline 4, 18, 44, 49, 50, 97, 100, 101, 103, 142  &  54
\\ \hline 24, 28, 30, 41, 56, 67, 87, 122, 135, 143  &  55
\\ \hline 89, 106  &  56
\\ \hline 11, 84, 91  &  57
\\ \hline 35, 144  &  58
\\ \hline 51  &  59
\\ \hline
\end{array}\\[-3ex]
\]
\caption{Maximal length of the chains of identical bits after the
  rounding bit (assuming the target precision is double precision)
  in the worst cases for $n$ from $3$ to $\nmax$.}
\label{worst-cases-double}
\end{table}
\newpage
Define a \emph{breakpoint} as the exact middle of two consecutive double
precision numbers. $\RNd(h+l)$ will be equal to $\RNd(x^n)$ if and only if
there is no breakpoint between $x^n$ and $h+l$.

The worst case obtained shows that if $x$ is a double-precision
number, and if $3 \leq n \leq \nmax$, then the significand $y$ of
$x^{51}$ is always at a distance larger than $2^{-113}$ from the
breakpoint~$\mu$ (see Figure~\ref{fig-1}) where the distance
$|y-\mu| \geq 2^{-(53+59+1)}=2^{-113}$.

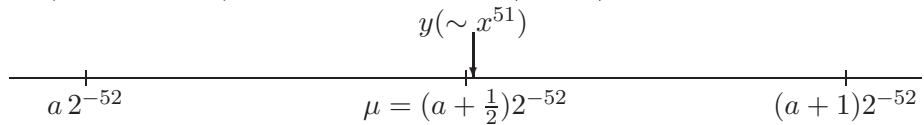
\begin{figure}[hb]
\begin{center}\small
\setlength{\unitlength}{1mm}
\begin{picture}(120,6)
\put(0,0){\line(1,0){120}}
\put(10,-1){\line(0,1){2}}\put(10,-4){\mb{$a\,2^{-52}$}}
\put(60,-1){\line(0,1){2}}\put(60,-4){\mb{$\mu=(a+\frac12)2^{-52}$}}
\put(110,-1){\line(0,1){2}}\put(110,-4){\mb{$(a+1)2^{-52}$}}
\put(61,6){\vector(0,-1){6}}\put(61,7){\mb{$y(\sim x^{51})$}}
\end{picture}
\end{center}
\caption{Position of the hardest to round case $y=x^{51}$ within rounding
  interval $[a2^{-52};(a+1)2^{-52}]$ with breakpoint
  $\mu=(a+\frac12)2^{-52}$, for significand defined by integer~$a$.}
\label{fig-1}
\end{figure}

We know that the significand of $h+l$ is within $2 \alpha_{max}$
from that of $x^n$, where $\alpha_{max}$ (as given by its binary
logarithm) is listed in Table~\ref{table-logpower-extended}. For all
values of $n$ less than or equal to $\nmax$, we have $2 \alpha_{max}
\leq 2^{-113}$, thus $\RNd(h+l) = \RNd(x^n)$. We therefore get the
following result:

\begin{theorem}
  \label{th-dbleextended-correct-rounding} If algorithm LogPower is run in
  double-extended precision, and if $3 \leq n \leq \nmax$, then $\RNd(h+l) =
  \RNd(x^n)$: Hence by rounding $h+l$ to the nearest double-precision
  number, we get a correctly rounded result.
\end{theorem}

\noindent
Now, two important remarks:
\begin{itemize}
\item We do not have the worst cases for $n > \nmax$, but from probabilistic
  arguments we strongly believe that the lengths of the largest chains of
  consecutive bits after the rounding bit will be of the same order of
  magnitude (i.e., around $50$) for some range of $n$ above $\nmax$.
  However, it is unlikely that we will be able to show correct rounding in
  double precision for values of $n$ larger than $1000$.

\item On an Intel Itanium processor, it is possible to directly add two
  double-extended precision numbers and round the result to
  double precision without a ``double rounding'' (i.e., without having an
  intermediate sum rounded to double-extended precision).  Hence
  Theorem~\ref{th-dbleextended-correct-rounding} can directly be used.
  It is worth being noticed that the draft revised standard IEEE 754-R
  (see \url{http://754r.ucbtest.org/}) includes the fma as well as rounding to any specific destination
format, independent of operand formats.

\end{itemize}

\section*{Conclusion}

It has been shown that the function $x^n$ can be calculated in time ${\cal
  O}(\log n)$ with correct rounding in double precision, employing
double-extended precision arithmetic, at least for the range $3 \leq n \leq
\nmax$. A fused multiply accumulate (fma) instruction is assumed available
for algorithm efficiency reasons; and to keep the analysis simple, it was
assumed that the input as well as the output are not subnormal numbers and
are below the overflow threshold.

A simpler, ${\cal O}(n)$ time algorithm, faster than the above for small
values of $n$, was also analyzed. However, its error analysis turned out to
be more complicated (and less rigorous), and also to be less accurate than
the other.

\bibliographystyle{plain}

\end{document}